\documentclass[10pt,oneside,leqno]{amsart}

\usepackage{amsmath,amsfonts,amssymb}
\usepackage{graphicx}
\usepackage[all]{xy}
\usepackage[english]{babel}
\usepackage[utf8x]{inputenc}
\usepackage{lscape}
\usepackage{url}
\usepackage{fullpage}

\newtheorem{thm}{Theorem}[section]

\newtheorem{lemma}[thm]{Lemma}
\newtheorem{conj}[thm]{Conjecture}
\theoremstyle{definition}

\newtheorem{rem}[thm]{Remark}

\newcommand{\N}{\mathbb{N}}
\newcommand{\Z}{\mathbb{Z}}
\newcommand{\Q}{\mathbb{Q}}
\newcommand{\R}{\mathbb{R}}
\newcommand{\C}{\mathbb{C}}

\newcommand{\I}{\mathbb{I}}
\newcommand{\st}{\;:\;}

\newcommand{\GL}{\mathrm{GL}}
\newcommand{\phit}[1]{\varphi^{#1}_{\mathbf{t}}}

\newcommand{\tempo}{\mathbf{t}}
\newcommand{\zero}{\mathbf{0}}

\newcommand{\sspace}{\cdot}

\newcommand{\opiccolo}[1]{\mathrm{o}\left(\left|#1\right|\right)}

\newcommand{\opiccolouno}{\mathrm{o}\left(1\right)}

\DeclareMathOperator{\im}{i}
\DeclareMathOperator{\imm}{im}
\DeclareMathOperator{\de}{d}
\DeclareMathOperator{\id}{id}

\DeclareMathOperator{\Nij}{Nij}
\DeclareMathOperator{\End}{End}

\DeclareMathOperator{\Span}{span}
\DeclareMathOperator{\Ad}{Ad}
\DeclareMathOperator{\ad}{ad}
\DeclareMathOperator{\rk}{rk}

\newcommand{\del}{\partial}
\newcommand{\delbar}{\overline{\del}}
\newcommand{\delbarstar}{\overline{\del}^*}

\title[The cohomologies of the Iwasawa manifold and deformations]
{The cohomologies of the Iwasawa manifold and of its small deformations}
\author{Daniele Angella}
\address{Dipartimento di Matematica ``Leonida Tonelli''\\
Universit\`{a} di Pisa \\
Largo Bruno Pontecorvo 5, 56127\\ 
Pisa, Italy}
\email{angella@mail.dm.unipi.it}

\keywords{Iwasawa manifold; cohomology; Bott-Chern; solvmanifold; deformations}
\thanks{This work was supported by GNSAGA of INdAM}
\subjclass[2010]{57T15; 53C15; 32G05}

\begin{document}

\vspace{-2cm}
\begin{minipage}[l]{9cm}
{\sffamily
  D. Angella,
  The cohomologies of the Iwasawa manifold and of its small deformations,
  \textsc{doi:} \texttt{10.1007/s12220-011-9291-z},
  to appear in {\itshape J. Geom. Anal.}.

\smallskip

  \begin{flushright}\begin{footnotesize}
  (The final publication is available at \url{www.springerlink.com}.)
  \end{footnotesize}\end{flushright}
}
\end{minipage}
\vspace{2cm}

\begin{abstract}
 We prove that, for some classes of complex nilmanifolds, the Bott-Chern cohomology is completely determined by the Lie algebra associated to the nilmanifold with the induced complex structure.
 We use these tools to compute the Bott-Chern and Aeppli cohomologies of the Iwasawa manifold and of its small deformations, completing the computations in \cite{schweitzer} by M. Schweitzer.
\end{abstract}

\maketitle

\section*{Introduction}

On compact oriented Riemannian manifolds, Hodge theory allows to compute cohomology solving systems of differential equations.
For \emph{nilmanifolds}, namely, compact quotients of connected simply-connected nilpotent Lie groups by co-compact discrete subgroups, K. Nomizu proved in \cite[Theorem 1]{nomizu} that the de Rham complex admits a finite-dimensional subcomplex, defined in Lie theoretic terms, as minimal model.
Furthermore, also the Dolbeault cohomology often reduces to the cohomology of the corresponding finite-dimensional complex of forms on the Lie algebra: this happens, for example, for holomorphically parallelizable complex structures, as proved by Y. Sakane in \cite[Theorem 1]{sakane}, or for rational complex structures, as proved by S. Console and A. Fino in \cite[Theorem 2]{console-fino}; see the surveys \cite{console} and \cite{rollenske-survey} and also \cite{rollenske} for more results in this direction.

Together with the Dolbeault cohomology, the \emph{Bott-Chern cohomology} provides an important tool to study the geometry of complex manifolds: it is the bi-graded algebra defined by
$$ H^{\bullet,\bullet}_{BC}(X) \;:=\; \frac{\ker\del\cap\ker\delbar}{\imm\del\delbar} \;,$$
$X$ being a complex manifold. It turns out that there is a Hodge theory also for the Bott-Chern cohomology (see, e.g., \cite[\S2]{schweitzer}) and therefore, if $X$ is compact, then the $\C$-vector space $H^{\bullet,\bullet}_{BC}(X)$ is finite-dimensional; furthermore, if $X$ is a compact K\"ahler manifold or, more in general, if $X$ is in Fujiki class $\mathcal{C}$, then the Dolbeault and Bott-Chern cohomologies coincide and they give a splitting for the de Rham cohomology algebra.

In this paper, we give some tools to compute the Bott-Chern cohomology ring of certain compact complex homogeneous manifolds.\\
More precisely, let $N=\left.\Gamma\right\backslash G$ be a nilmanifold endowed with a $G$-left-invariant complex structure $J$ and denote the Lie algebra naturally associated to $G$ by $\mathfrak{g}$ and its complexification by $\mathfrak{g}_\C:=\mathfrak{g}\otimes_\R\C$.
Dealing with $G$-left-invariant objects on $N$, we mean objects whose pull-back to $G$ is invariant for the left-action of $G$ on itself.
For any $k\in\N$ and $p,q\in\N$, we denote the space of smooth global sections of the bundle of real $k$-forms (respectively, complex $k$-forms, $(p,q)$-forms) on $N$ by the symbol $\wedge^kN$ (respectively, $\wedge^k(N;\C)$, $\wedge^{p,q}N$).
For any $p,q\in\N$, the $(p,q)$-th Bott-Chern cohomology group $H^{p,q}_{BC}(N)$ of $N$ is computed as the cohomology of the complex
$$
\wedge^{p-1,q-1}N \stackrel{\del\delbar}{\longrightarrow} \wedge^{p,q}N \stackrel{\de}{\longrightarrow} \wedge^{p+q+1}(N;\C) \;.
$$
Restricting to $G$-left-invariant forms $\wedge^{\bullet,\bullet}_\text{inv}N\simeq \wedge^{\bullet,\bullet}\mathfrak{g}_\C^*$ on $N$, one has the subcomplex
\begin{equation}\label{eq:bc-subcomplex}
\begin{gathered}
\xymatrix{
\wedge^{p-1,q-1}\mathfrak{g}_\C^* \ar[r]^{\del\delbar} \ar@{^{(}->}[d] & \wedge^{p,q}\mathfrak{g}_\C^* \ar[r]^{\de} \ar@{^{(}->}[d] & \wedge^{p+q+1}\mathfrak{g}_\C^* \ar@{^{(}->}[d] \\
\wedge^{p-1,q-1}N \ar[r]^{\del\delbar} & \wedge^{p,q}N \ar[r]^{\de} & \wedge^{p+q+1}(N;\C)
}
\end{gathered}
\;.
\end{equation}
We prove the following theorem, which extends the results in \cite[Theorem 1]{sakane}, \cite[Main Theorem]{cordero-fernandez-gray-ugarte-2000}, \cite[Theorem 1, Theorem 2, Remark 4]{console-fino}, \cite[Theorem 1.10]{rollenske}, see also \cite{console, rollenske-survey}, to the Bott-Chern case.

\smallskip
\noindent {\bfseries Theorem} (see Theorem \ref{thm:bc-invariant-cor} and Theorem \ref{thm:bc-invariant-open}){\bfseries .}
{\itshape
 Let $N=\left.\Gamma\right\backslash G$ be a nilmanifold endowed with a $G$-left-invariant complex structure $J$ and denote the Lie algebra naturally associated to $G$ by $\mathfrak{g}$.
Then, for every $p,q\in\N$, the injective homomorphism in cohomology
$$
i\colon\frac{\ker\left(\de\colon\wedge^{p,q}\mathfrak{g}_\C^*\to\wedge^{p+q+1}\mathfrak{g}_\C^*\right)}{\imm\left(\del\delbar\colon\wedge^{p-1,q-1}\mathfrak{g}_\C^*\to\wedge^{p,q}\mathfrak{g}_\C^*\right)} \; \hookrightarrow \; H^{p,q}_{BC}(N)
$$
induced by \eqref{eq:bc-subcomplex} is an isomorphism, provided one of the following conditions holds:
\begin{itemize}
 \item $N$ is holomorphically parallelizable;
 \item $J$ is an Abelian complex structure;
 \item $J$ is a nilpotent complex structure;
 \item $J$ is a rational complex structure;
 \item $\mathfrak{g}$ admits a torus-bundle series compatible with $J$ and with the rational structure induced by $\Gamma$.
\end{itemize}
Moreover, the property of $i$ being an isomorphism is open in the space of all $G$-left-invariant complex structures on $N$.
}
\smallskip

\noindent Similar results are obtained for a certain class of solvmanifolds.

Then, we use these tools to explicitly compute the Bott-Chern cohomology for a three-dimensional holomorphically parallelizable nilmanifold, the so called \emph{Iwasawa manifold} (see, e.g., \cite{fernandez-gray}) and for its small deformations. The Iwasawa manifold is one of the simplest example of non-K\"ahler manifold: indeed, being non-formal, see \cite[page 158]{fernandez-gray}, it admits no K\"ahler structure; it has been studied by several authors as a fruitful source of interesting behaviors: see, for example, \cite{fernandez-gray, nakamura, alessandrini-bassanelli, bassanelli, ye, angella-tomassini}.\\
As regards the Dolbeault cohomology, I. Nakamura, in \cite{nakamura}, already computed the Hodge numbers for the Iwasawa manifold and for its small deformations: he used these computations to prove that the Hodge numbers are not invariant under small deformations, \cite[Theorem 2]{nakamura} (compare also \cite[\S4]{ye}), and that small deformations of a holomorphically parallelizable complex structure are not necessarily holomorphically parallelizable, \cite[page 86]{nakamura} (compare also \cite[Theorem 5.1, Corollary 5.2]{rollenske-jems}). The Bott-Chern cohomology and its computation for the Iwasawa manifold appeared in the work by M. Schweitzer, see \cite[\S1.c]{schweitzer}, while the computations of the Bott-Chern cohomology for its small deformations were there announced but never written.\\
In particular, the computations in \S\ref{sec:bott-chern-iwasawa} show that one can use the Bott-Chern cohomology to get a finer classification of small deformations of the Iwasawa manifold than using the Dolbeault cohomology as in \cite[page 96]{nakamura}.

\medskip

\noindent{\itshape Acknowledgments.}
The author would like to thank Adriano Tomassini and Jean-Pierre Demailly for their constant encouragement, their support and for many useful conversations. He would like also to thank Institut Fourier, Université de Grenoble \textsc{i}, for its warm hospitality.
Very interesting conversations with S\"onke Rollenske at \textsc{cirm} in Luminy and with Greg Kuperberg at Institut Fourier in Grenoble gave great motivations for looking at further results on this subject. Many thanks to S\"onke are due also for his comments and remarks which improved the presentation of this paper.
The author is very grateful to the anonymous referee for his/her careful reading and for many suggestions and remarks that highly improved the presentation of the paper.

\section{The Bott-Chern and Aeppli cohomologies of a complex manifold}\label{sec:bott-chern-cohomology}

\subsection{The Bott-Chern cohomology}\label{subsec:hodge-theory-bott-chern}
Let $X$ be a compact complex manifold of complex dimension $n$ and denote its complex structure by $J$. The \emph{Bott-Chern cohomology} groups are defined, for $p,q\in\N$, as
$$ H^{p,q}_{BC}(X) \;:=\; \frac{\ker \left(\de\colon\wedge^{p,q}X\to\wedge^{p+q+1}(X;\C)\right)}{\imm\left(\del\delbar\colon\wedge^{p-1,q-1}X\to\wedge^{p+1,q+1}X\right)} \;.$$
Note that, for every $p,q\in\N$, the conjugation induces an isomorphism
$$ H^{p,q}_{BC}(X)\;\simeq\; H^{q,p}_{BC}(X) \;,$$
unlike in the case of the Dolbeault cohomology groups.\\
For every $k\in\N$ and for every $p,q\in\N$, one has the natural maps
$$ \bigoplus_{r+s=k}H^{r,s}_{BC}(X)\to H^{k}_{dR}(X;\C) \qquad \text{ and } \qquad H^{p,q}_{BC}(X)\to H^{p,q}_{\delbar}(X) \;.$$
In general, these maps are neither injective nor surjective, see, e.g., the examples in \cite[\S1.c]{schweitzer} or \S\ref{sec:bott-chern-iwasawa}; if $X$ satisfies the $\del\delbar$-Lemma (for example, if $X$ admits a K\"ahler structure or if $X$ is in Fujiki class $\mathcal{C}$), then the above maps are isomorphisms, see \cite[Remark 5.16]{deligne-griffiths-morgan-sullivan}.

We collect here some results on Hodge theory for the Bott-Chern cohomology, referring to \cite[\S2]{schweitzer} (see also \cite[\S5]{bigolin}).
Fix $g$ a $J$-Hermitian metric on $X$ and define the $4$-th order self-adjoint elliptic differential operator
$$ \tilde\Delta_{BC} \;:=\; \del\delbar\delbarstar\del^*+\delbarstar\del^*\del\delbar+\delbarstar\del\del^*\delbar+\del^*\delbar\delbarstar\del+\delbarstar\delbar+\del^*\del \;, $$
see \cite[Proposition 5]{kodaira-spencer-3} and also \cite[\S2.b]{schweitzer}, \cite[\S5.1]{bigolin}. Given $u\in\wedge^{p,q}X$, one has that
$$ \tilde\Delta_{BC}\,u\;=\;0 \qquad \Leftrightarrow \qquad
\left\{
\begin{array}{rcl}
 \delbar u &=& 0 \\[5pt]
 \del u &=& 0 \\[5pt]
 \delbarstar\del^* u &=& 0
\end{array}
\right. \;;
$$
moreover, being $\tilde\Delta_{BC}$ a self-adjoint elliptic differential operator, the following result holds, see, e.g., \cite[page 450]{kodaira}.
\begin{thm}[{\cite[Th\'eor\`eme 2.2]{schweitzer}, \cite[Corollaire 2.3]{schweitzer}}]
 Let $X$ be a compact complex manifold and denote its complex structure by $J$; fix $g$ a $J$-Hermitian metric on $X$. Then there are an orthogonal decomposition
$$ \wedge^{\bullet,\bullet}X \;=\; \ker\tilde\Delta_{BC} \,\oplus\, \imm\del\delbar \,\oplus\, \left(\imm\del^*\,+\,\imm\delbarstar\right) $$
and an isomorphism
$$ H^{\bullet,\bullet}_{BC}(X) \;\simeq\; \ker\tilde\Delta_{BC} \;.$$
In particular, its Bott-Chern cohomology groups are finite-dimensional $\C$-vector spaces.
\end{thm}

\subsection{The Aeppli cohomology}\label{subsec:hodge-theory-aeppli}
Let $X$ be a compact complex manifold of complex dimension $n$ and denote its complex structure by $J$.
For $p,q\in\N$, one defines the \emph{Aeppli cohomology} group $H^{p,q}_{A}(X)$ as
$$ H^{p,q}_{A}(X) \;:=\; \frac{\ker \left(\del\delbar\colon\wedge^{p,q}X\to\wedge^{p+1,q+1}X\right)}{\left(\imm\left(\del\colon\wedge^{p-1,q}X\to\wedge^{p,q}X\right)\right)\,+\,\left(\imm\left(\delbar\colon\wedge^{p,q-1}X\to\wedge^{p,q}X\right)\right)} \;.$$
As for the Bott-Chern cohomology, the conjugation induces an isomorphism
$$ H^{p,q}_{A}(X)\;\simeq\; H^{q,p}_{A}(X) $$
for every $p,q\in\N$.\\
Furthermore, for every $k\in\N$ and for every $p,q\in\N$, one has the natural maps
$$ H^{k}_{dR}(X;\C)\to \bigoplus_{r+s=k}H^{r,s}_{A}(X) \qquad \text{ and } \qquad  H^{p,q}_{\delbar}(X)\to H^{p,q}_{A}(X) \;,$$
which are, in general, neither injective nor surjective; once again, the maps above are isomorphisms if $X$ satisfies the $\del\delbar$-Lemma, \cite[Remark 5.16]{deligne-griffiths-morgan-sullivan}, and hence, in particular, if $X$ admits a K\"ahler structure or if $X$ is in Fujiki class $\mathcal{C}$.

\begin{rem}
 On a K\"ahler manifold $X$, the fundamental $2$-form $\omega$ associated to the metric defines a non-zero class in $H^2_{dR}(X;\R)$. For general Hermitian manifolds, special classes of metrics are often defined in terms of closedness of powers of $\omega$ (e.g., a Hermitian metric on a complex manifold of complex dimension $n$ is said \emph{balanced} if $\de\omega^{n-1}=0$, \emph{pluriclosed} if $\del\delbar\omega=0$, \emph{astheno-K\"ahler} if $\del\delbar\omega^{n-2}=0$, \emph{Gauduchon} if $\del\delbar\omega^{n-1}=0$), so they define classes in the Bott-Chern or Aeppli cohomology groups. It would be interesting to see if these classes can play a role similar to the one played by the K\"ahler class in some contexts.
\end{rem}

We refer to \cite[\S2.c]{schweitzer} for the results that follows (for a hypercohomology interpretation of the Bott-Chern and Aeppli cohomologies and its applications, see \cite[\S VI.12.1]{demailly-agbook}, \cite[\S4]{schweitzer}, \cite[\S3.2, \S3.5]{kooistra}).\\
Fixed a $J$-Hermitian metric $g$ on $X$ and defined the $4$-th order self-adjoint elliptic differential operator
$$ \tilde\Delta_{A} \;:=\; \del\del^*+\delbar\delbarstar+\delbarstar\del^*\del\delbar+\del\delbar\delbarstar\del^*+\del\delbarstar\delbar\del^*+\delbar\del^*\del\delbarstar \;, $$
one has an orthogonal decomposition
$$ \wedge^{\bullet,\bullet}X \;=\; \ker\tilde\Delta_{A} \,\oplus\, \left(\imm\del \,+\, \imm\delbar\right)
\,\oplus\, \imm\left(\del\delbar\right)^* $$
from which one gets an isomorphism
$$ H^{\bullet,\bullet}_{A}(X) \;\simeq\; \ker\tilde\Delta_{A} \;;$$
this proves that the Aeppli cohomology groups of a compact complex manifold are finite-dimensional $\C$-vector spaces.

In fact, for any $p,q\in\N$, one has that the Hodge-$*$-operator associated to a $J$-Hermitian metric induces an isomorphism
$$ H^{p,q}_{BC}(X) \simeq H^{n-q,n-p}_{A}(X) $$
between the Bott-Chern and the Aeppli cohomologies.

\section{Some results on cohomology computation}\label{sec:cohomology-computation}
In this section, we collect some results about cohomology computation for nilmanifolds and solvmanifolds. Using these tools, one recovers the de Rham, Dolbeault, Bott-Chern and Aeppli cohomologies for the Iwasawa manifold and for its small deformations, see \S\ref{sec:derham-iwasawa}, \S\ref{sec:dolbeault-iwasawa} and \S\ref{sec:bott-chern-iwasawa}.

\smallskip

Let $X=\left. \Gamma \right\backslash G$ be a solvmanifold, that is, a compact quotient of the connected simply-connected solvable Lie group $G$ by a discrete and co-compact subgroup $\Gamma$; the Lie algebra naturally associated to $G$ will be denoted by $\mathfrak{g}$ and its complexification by $\mathfrak{g}_\C:=\mathfrak{g}\otimes_\R\C$. Dealing with $G$-left-invariant objects on $X$, we mean objects on $X$ obtained by objects on $G$ that are invariant under the action of $G$ on itself given by left-translation.

A $G$-left-invariant complex structure $J$ on $X$ is uniquely determined by a linear complex structure $J$ on $\mathfrak{g}$ satisfying the integrability condition
$$ \forall x,y\in\mathfrak{g}\;,\qquad \Nij_J(x,y)\;:=\;\left[x,\,y\right]+J\left[Jx,\,y\right]+J\left[x,\,Jy\right]-\left[Jx,\,Jy\right]\;=\;0 \;,$$
see \cite[Theorem 1.1]{newlander-nirenberg}. Therefore, the set of $G$-left-invariant complex structures on $X$ is given by
$$ \mathcal{C}\left(\mathfrak{g}\right) := \left\{ J\in\End\left(\mathfrak{g}\right) \st J^2=-\id_{\mathfrak{g}} \;\text{ and }\;\Nij_J=0 \right\} \;.$$

Recall that the exterior differential $\de$ on $X$ can be written using only the action of $\Gamma(X;\,TX)$ on $\mathcal{C}^\infty(X)$ and the Lie bracket on the Lie algebra of vector fields on $X$. One has that the complex $\left(\wedge^\bullet\mathfrak{g}^*,\,\de\right)$ is isomorphic, as a differential complex, to the differential subcomplex $\left(\wedge^\bullet_{\text{inv}}X,\,\de\lfloor_{\wedge^\bullet_{\text{inv}}X}\right)$ of $\left(\wedge^\bullet X,\,\de\right)$ given by the $G$-left-invariant forms on $X$.\\
If a $G$-left-invariant complex structure on $X$ is given, then one has also the double complex $\left(\wedge^{\bullet,\bullet}\mathfrak{g}_{\C}^*,\,\del,\,\delbar\right)$, which is isomorphic, as a double complex, to the double subcomplex $\left(\wedge^{\bullet,\bullet}_{\text{inv}}X,\,\del\lfloor_{\wedge^{\bullet,\bullet}_{\text{inv}}X},\,\delbar\lfloor_{\wedge^{\bullet,\bullet}_{\text{inv}}X}\right)$ of $\left(\wedge^{\bullet,\bullet}X,\,\del,\,\delbar\right)$ given by the $G$-left-invariant forms on $X$.\\
Lastly, given a $G$-left-invariant complex structure on $G$ and fixed $p,q\in\N$, one has also the following complexes and the following maps of complexes:
\begin{equation}\label{eq:bc-complessi}
\begin{gathered}
\xymatrix{
\wedge^{p-1,q-1}\mathfrak{g}_\C^* \ar[r]^{\del\delbar} \ar[d]^{\simeq} & \wedge^{p,q}\mathfrak{g}_\C^* \ar[r]^{\de} \ar[d]^{\simeq} & \wedge^{p+q+1}\mathfrak{g}_\C^* \ar[d]^{\simeq} \\
\wedge^{p-1,q-1}_{\text{inv}}X \ar[r]^{\del\delbar} \ar@{^{(}->}[d]^{i} & \wedge^{p,q}_{\text{inv}}X \ar[r]^{\de} \ar@{^{(}->}[d]^{i} & \wedge^{p+q+1}_{\text{inv}}(X;\C) \ar@{^{(}->}[d]^{i} \\
\wedge^{p-1,q-1}X \ar[r]^{\del\delbar} & \wedge^{p,q}X \ar[r]^{\de} & \wedge^{p+q+1}(X;\C)
}
\end{gathered}
\;;
\end{equation}
the same can be repeated for the complex used to define the $(p,q)$-th Aeppli cohomology group.

For $\star\in\left\{\delbar,\,\del,\,BC,\,A\right\}$ and $\mathbb{K}\in\left\{\R,\C\right\}$, we will write $H^\bullet_{dR}\left(\mathfrak{g};\mathbb{K}\right)$ and $H^{\bullet,\bullet}_{\star}\left(\mathfrak{g}_\C\right)$ to denote the cohomology groups of the corresponding complexes of forms on $\mathfrak{g}$, equivalently, of $G$-left-invariant forms on $X$. The rest of this section is devoted to the problem whether these cohomologies are isomorphic to the corresponding cohomologies on $X$.

\subsection{Classical results on computation for the de Rham and Dolbeault cohomologies}\label{subsec:cohomology-computation-derham-dolbeault}

One has the following theorem by K. Nomizu, saying that the de Rham cohomology of a nilmanifold can be computed as the cohomology of the subcomplex of left-invariant forms.

\begin{thm}[{\cite[Theorem 1]{nomizu}}]\label{thm:nomizu}
 Let $N=\left. \Gamma \right\backslash G$ be a nilmanifold and denote the Lie algebra naturally associated to $G$ by $\mathfrak{g}$. The complex $\left(\wedge^{\bullet}\mathfrak{g}^*,\de\right)$ is a minimal model for $N$. In particular, the map of complexes $\left(\wedge^\bullet\mathfrak{g}^*,\,\de\right)\to\left(\wedge^\bullet N,\,\de\right)$ is a quasi-isomorphism, that is, it induces an isomorphism in cohomology:
$$ i\colon H^{\bullet}_{dR}(\mathfrak{g};\R) \stackrel{\simeq}{\longrightarrow} H^\bullet_{dR}(N;\R) \;.$$
\end{thm}

\noindent The proof rests on an inductive argument, which can be performed since every nilmanifold can be seen as a principal torus-bundle over a lower dimensional nilmanifold.\\
A similar result holds also in the case of completely-solvable solvmanifolds and has been proved by A. Hattori in \cite[Corollary 4.2]{hattori}, as a consequence of the Mostow Structure Theorem (see also \cite[Chapter 3]{tralle-oprea}). We recall that a solvmanifold $X=\left.\Gamma\right\backslash G$ is said \emph{completely-solvable} if, for any $g\in G$, all the eigenvalues of $\Ad g$ are real, equivalently, if, for any $X\in\mathfrak{g}$, all the eigenvalues of $\ad X$ are real.

\begin{thm}[{\cite[Corollary 4.2]{hattori}}]
 Let $X=\left. \Gamma \right\backslash G$ be a completely-solvable solvmanifold and denote the Lie algebra naturally associated to $G$ by $\mathfrak{g}$. Then the map of complexes $\left(\wedge^\bullet\mathfrak{g}^*,\,\de\right)\to\left(\wedge^\bullet X,\,\de\right)$ is a quasi-isomorphism, that is, it induces an isomorphism in cohomology:
$$ i\colon H^{\bullet}_{dR}(\mathfrak{g};\R) \stackrel{\simeq}{\longrightarrow} H^\bullet_{dR}(X;\R) \;.$$
\end{thm}

\noindent In general, for non-completely-solvable solvmanifolds the map of complexes $\left(\wedge^\bullet\mathfrak{g}^*,\,\de\right)\to\left(\wedge^\bullet X,\,\de\right)$ is not necessarily a quasi-isomorphism, as the example in \cite[Corollary 4.2, Remark 4.3]{debartolomeis-tomassini} shows (for some results about the de Rham cohomology of solvmanifolds, see \cite{console-fino-2}).

\medskip

Considering nilmanifolds endowed with certain $G$-left-invariant complex structures, there are similar results also for the Dolbeault cohomology, see, e.g.,  \cite{console} and \cite{rollenske-survey} for surveys on the known results. (Some results about the Dolbeault cohomology of solvmanifolds have been recently proved by H. Kasuya, see \cite{kasuya}.)

\begin{thm}[{\cite[Theorem 1]{sakane}, \cite[Main Theorem]{cordero-fernandez-gray-ugarte-2000}, \cite[Theorem 2, Remark 4]{console-fino}, \cite[Theorem 1.10]{rollenske}}]\label{thm:dolbeault-invariant}
 Let $N=\left. \Gamma \right\backslash G$ be a nilmanifold and denote the Lie algebra naturally associated to $G$ by $\mathfrak{g}$. Let $J$ be a $G$-left-invariant complex structure on $N$.
Then, for every $p\in\N$, the map of complexes
\begin{equation}\label{eq:dolb-complessi}
\left(\wedge^{p,\bullet}\mathfrak{g}_{\C}^*,\,\delbar\right) 
\hookrightarrow\left(\wedge^{p,\bullet}N,\,\delbar\right)
\end{equation}
is a quasi-isomorphism, hence
$$
i \colon H^{\bullet,\bullet}_{\delbar}\left(\mathfrak{g}_\C\right) \stackrel{\simeq}{\longrightarrow} H^{\bullet,\bullet}_{\delbar}(N) \;,$$
provided one of the following conditions holds:
\begin{itemize}
 \item $N$ is holomorphically parallelizable, see \cite[Theorem 1]{sakane};
 \item $J$ is an \emph{Abelian} complex structure (i.e., $\left[Jx,\,Jy\right]=\left[x,\,y\right]$ for any $x,y\in\mathfrak{g}$), see \cite[Remark 4]{console-fino};
 \item $J$ is a \emph{nilpotent} complex structure (i.e., there is a $G$-left-invariant co-frame $\left\{\omega^1,\ldots,\omega^n\right\}$ for $\left(T^{1,0}N\right)^*$ with respect to which the structure equations of $N$ are of the form
$$ \de\omega^j\;=\; \sum_{h<k<j}A_{hk}^j\,\omega^h\wedge\omega^k+\sum_{h,k<j}B_{h k}^j\,\omega^h\wedge\bar\omega^k $$
with $\left\{A_{hk}^j,\,B_{h k}^j\right\}_{j,h,k}\subset\C$), see \cite[Main Theorem]{cordero-fernandez-gray-ugarte-2000};
 \item $J$ is a \emph{rational} complex structure (i.e., $J\left(\mathfrak{g}_\Q\right)\subseteq \mathfrak{g}_\Q$ where $\mathfrak{g}_\Q$ is the rational structure for $\mathfrak{g}$ ---that is, a $\Q$-vector space such that $\mathfrak{g}=\mathfrak{g}_\Q\otimes_\Q\R$--- induced by $\Gamma$), see \cite[Theorem 2]{console-fino};
 \item $\mathfrak{g}$ admits a torus-bundle series compatible with $J$ and with the rational structure induced by $\Gamma$, see \cite[Theorem 1.10]{rollenske}.
\end{itemize}
\end{thm}

\medskip

We recall also the following theorem by S. Console and A. Fino.
\begin{thm}[{\cite[Theorem 1]{console-fino}}]\label{thm:dolbeault-invariant-open}
 Let $N=\left.\Gamma\right\backslash G$ be a nilmanifold and denote the Lie algebra naturally associated to $G$ by $\mathfrak{g}$. Given any $G$-left-invariant complex structure $J$ on $N$, the map of complexes \eqref{eq:dolb-complessi}
induces an injective homomorphism $i$ in cohomology, \cite[Lemma 9]{console-fino}):
$$ i\colon H^{\bullet,\bullet}_{\delbar}\left(\mathfrak{g}_\C\right) \hookrightarrow H^{\bullet,\bullet}_{\delbar}(N) \;.$$
 Let $\mathcal{U}\subseteq\mathcal{C}(\mathfrak{g})$ be the subset containing the $G$-left-invariant complex structures $J$ on $N$ such that the inclusion $i$ is an isomorphism:
$$ \mathcal{U} \;:=\; \left\{ J\in\mathcal{C}\left(\mathfrak{g}\right) \st i\colon H^{\bullet,\bullet}_{\delbar}\left(\mathfrak{g}_\C\right)\stackrel{\simeq}{\hookrightarrow}H^{\bullet,\bullet}_{\delbar}(N)\right\} \;\subseteq\; \mathcal{C}\left(\mathfrak{g}\right) \;. $$
Then $\mathcal{U}$ is an open set in $\mathcal{C}\left(\mathfrak{g}\right)$.
\end{thm}

\noindent The strategy of the proof consists in proving that the dimension of the orthogonal of $H^{\bullet,\bullet}_{\delbar}\left(\mathfrak{g}_\C\right)$ in $H^{\bullet,\bullet}_{\delbar}(N)$ with respect to a given $J$-Hermitian $G$-left-invariant metric on $N$ is an upper-semi-continuous function in $J\in\mathcal{C}\left(\mathfrak{g}\right)$ and thus if it is zero for a given $J\in\mathcal{C}\left(\mathfrak{g}\right)$, then it remains equal to zero in an open neighborhood of $J$ in $\mathcal{C}\left(\mathfrak{g}\right)$.\\
We will use the same argument in proving Theorem \ref{thm:bc-invariant-open}, which is a slight modification of the previous result in the case of the Bott-Chern cohomology.

\medskip

The aforementioned results suggest the following conjecture.

\begin{conj}[{\cite[Conjecture 1]{rollenske-survey}; see also \cite[page 5406]{cordero-fernandez-gray-ugarte-2000}, \cite[page 112]{console-fino}}]\label{conj:dolbeault}
 Let $N=\left. \Gamma \right\backslash G$ be a nilmanifold endowed with a $G$-left-invariant complex structure $J$ and denote the Lie algebra naturally associated to $G$ by $\mathfrak{g}$.
Then the maps of complexes \eqref{eq:dolb-complessi} are quasi-isomorphisms, that is, they induce an isomorphism in cohomology:
$$ i\colon H^{\bullet,\bullet}_{\delbar}\left(\mathfrak{g}_\C\right) \stackrel{\simeq}{\longrightarrow} H^{\bullet,\bullet}_{\delbar}(N) \;.$$
\end{conj}

\noindent Note that, since $i$ is always injective by \cite[Lemma 9]{console-fino}, this is equivalent to ask that
$$ \dim_\C \left(H^{\bullet,\bullet}_{\delbar}\left(\mathfrak{g}_\C\right)\right)^\perp=0 \;,$$
where the orthogonality is meant with respect to the scalar product induced by a given $J$-Hermitian $G$-left-invariant metric $g$ on $N$.

\subsection{Some results on computation for the Bott-Chern cohomology}\label{subsec:cohomology-computation-bott-chern}
We prove here some results about Bott-Chern cohomology computation for nilmanifolds and solvmanifolds.

The first result is a slight modification of \cite[Lemma 9]{console-fino} proved by S. Console and A. Fino for the Dolbeault cohomology: we repeat here their argument in the case of the Bott-Chern cohomology.

\begin{lemma}\label{lemma:inj}
 Let $X=\left. \Gamma \right\backslash G$ be a solvmanifold endowed with a $G$-left-invariant complex structure $J$ and denote the Lie algebra naturally associated to $G$ by $\mathfrak{g}$.
 The map of complexes \eqref{eq:bc-complessi} induces an injective homomorphism between cohomology groups
$$ i\colon H^{\bullet,\bullet}_{BC}\left(\mathfrak{g}_\C\right) \hookrightarrow H^{\bullet,\bullet}_{BC}(X) \;.$$
\end{lemma}

\begin{proof}
 Fix $p,q\in\N$. Let $g$ be a $J$-Hermitian $G$-left-invariant metric on $X$ and consider the induced scalar product $\left\langle\left.\sspace\right|\sspace\right\rangle$ on $\wedge^{\bullet,\bullet} X$. Hence both $\del$, $\delbar$ and their adjoints $\del^*$, $\delbarstar$ preserve the $G$-left-invariant forms on $X$ and therefore also $\tilde\Delta_{BC}$ does. In such a way, we get a Hodge decomposition also at the level of $G$-left-invariant forms:
\begin{eqnarray*}
\wedge^{p,q}\mathfrak{g}_\C^* &=& \ker\tilde\Delta_{BC}\lfloor_{\wedge^{p,q}\mathfrak{g}_\C^*} \,\oplus\, \imm\del\delbar\lfloor_{\wedge^{p-1,q-1}\mathfrak{g}_\C^*} \\[5pt]
 && \oplus\, \left(\imm\del^*\lfloor_{\wedge^{p+1,q}\mathfrak{g}_\C^*} \,+\, \imm\delbarstar\lfloor_{\wedge^{p,q+1}\mathfrak{g}_\C^*}\right) \;.
\end{eqnarray*}
 Now, take $[\omega]\in H^{p,q}_{BC}\left(\mathfrak{g}_\C\right)$ such that $i\left[\omega\right]=0$ in $H^{p,q}_{BC}(X)$, that is, $\omega$ is a $G$-left-invariant $(p,q)$-form on $X$ and there exists a (possibly non-$G$-left-invariant) $(p-1,q-1)$-form $\eta$ on $X$ such that $\omega=\del\delbar\,\eta$. Up to zero terms in $H^{p,q}_{BC}\left(\mathfrak{g}_\C\right)$, we may assume that $\eta\in \left(i\left(\wedge^{p,q}\mathfrak{g}_\C^*\right)\right)^\perp\subseteq \wedge^{p,q}X$. Therefore, since $\delbarstar\del^*\del\delbar\eta$ is a $G$-left-invariant form (being $\del\delbar\eta$ a $G$-left-invariant form), we have that
$$ 0=\left\langle \left. \delbarstar\del^*\del\delbar\eta \right| \eta\right\rangle = \left\|\del\delbar\eta\right\|^2=\left\|\omega\right\|^2 $$
and therefore $\omega=0$.
\end{proof}

The second general result says that, if the Dolbeault and de Rham cohomologies of a solvmanifold are computed using just left-invariant forms, then also the Bott-Chern cohomology is computed using just left-invariant forms. The idea of the proof is inspired by \cite[\S1.c]{schweitzer}, where a similar argument is used to explicitly compute the Bott-Chern cohomology in the special case of the Iwasawa manifold.

\begin{thm}\label{thm:bc-invariant}
 Let $X=\left.\Gamma\right\backslash G$ be a solvmanifold endowed with a $G$-left-invariant complex structure $J$ and denote the Lie algebra naturally associated to $G$ by $\mathfrak{g}$. Suppose that
$$ i\colon H^\bullet_{dR}(\mathfrak{g};\C)\stackrel{\simeq}{\hookrightarrow} H^\bullet_{dR}(X;\C) \qquad \text{ and } \qquad i\colon H^{\bullet,\bullet}_{\delbar}\left(\mathfrak{g}_\C\right)\stackrel{\simeq}{\hookrightarrow} H^{\bullet,\bullet}_{\delbar}(X) \;.$$
Then also
$$ i\colon H^{\bullet,\bullet}_{BC}\left(\mathfrak{g}_\C\right) \stackrel{\simeq}{\hookrightarrow}H^{\bullet,\bullet}_{BC}(X) \;. $$
\end{thm}

\begin{proof}
 Fix $p,q\in\N$. We prove the theorem as a consequence of the following steps.
\smallskip
\paragraph{\bfseries Step 1} {\itshape We may reduce to study if $\frac{\imm\de\cap\wedge^{p,q}X}{\imm\del\delbar}$ can be computed using just $G$-left-invariant forms.}\\
Indeed, we have the exact sequence
$$ 0 \to \frac{\imm\de\cap\wedge^{p,q}X}{\imm\del\delbar} \to H^{p,q}_{BC}(X) \to H^{p+q}_{dR}(X;\C) $$
and, by hypothesis, $H^\bullet_{dR}(X;\C)$ can be computed using just $G$-left-invariant forms.
\smallskip
 \paragraph{\bfseries Step 2} {\itshape Under the hypothesis that the Dolbeault cohomology is computed using just $G$-left-invariant forms, if $\psi$ is a $G$-left-invariant $\delbar$-closed form then every solution $\phi$ of $\delbar\phi=\psi$ is $G$-left-invariant up to $\delbar$-exact terms.}\\
Indeed, since $[\psi]=0$ in $H^{\bullet,\bullet}_{\delbar}(X)$, there is a $G$-left-invariant form $\alpha$ such that $\psi=\delbar\alpha$. Hence, $\phi-\alpha$ defines a class in $H^{\bullet,\bullet}_{\delbar}(X)$ and hence $\phi-\alpha$ is $G$-left-invariant up to a $\delbar$-exact form, and so $\phi$ is.
\smallskip
 \paragraph{\bfseries Step 3} {\itshape Under the hypothesis that the Dolbeault cohomology is computed using just $G$-left-invariant forms, the space $\frac{\imm\de\cap\wedge^{p,q}X}{\imm\del\delbar}$ can be computed using just $G$-left-invariant forms.}\\
Consider
\begin{equation}\label{eq:proof-step-3}
\omega^{p,q}\;=\;\de\eta\mod\imm\del\delbar\;\in\;\frac{\imm\de\cap\wedge^{p,q}X}{\imm\del\delbar} \;.
\end{equation}
Decomposing $\eta=:\sum_{p,q}\eta^{p,q}$ in pure-type components, the equality \eqref{eq:proof-step-3} is equivalent to the system
$$
\left\{
\begin{array}{cccccccc}
 && \del\eta^{p+q-1,0} &=&0 & \mod \imm\del\delbar && \\[5pt]
\delbar\eta^{p+q-\ell,\ell-1} &+& \del\eta^{p+q-\ell-1,\ell} &=& 0 &\mod\imm\del\delbar & \text{ for } & \ell\in\{1,\ldots,q-1\} \\[5pt]
\delbar\eta^{p,q-1} &+& \del\eta^{p-1,q} &=& \omega^{p,q} & \mod\imm\del\delbar && \\[5pt]
\delbar\eta^{\ell,p+q-\ell-1} &+& \del\eta^{\ell-1,p+q-\ell} &=& 0 &\mod\imm\del\delbar & \text{ for } & \ell\in\{1,\ldots,p-1\} \\[5pt]
\delbar\eta^{0,p+q-1} &&&=&0 &\mod\imm\del\delbar &&
\end{array}
\right. \;.
$$
Applying several times {\itshape Step 2}, we may suppose that, for $\ell\in\{0,\ldots, p-1\}$, the forms $\eta^{\ell,p+q-\ell-1}$ are $G$-left-invariant: indeed, they are $G$-left-invariant up to $\delbar$-exact terms, but $\delbar$-exact terms give no contribution in the system, since it is modulo $\imm\del\delbar$. Analogously, using the conjugate version of {\itshape Step 2}, we may suppose that, for $\ell\in\{0,\ldots,q-1\}$, the forms $\eta^{p+q-\ell-1,\ell}$ are $G$-left-invariant. Then we may suppose that $\omega^{p,q}=\delbar\eta^{p,q-1}+\del\eta^{p-1,q}$ is $G$-left-invariant.
\end{proof}

As a corollary of Theorem \ref{thm:nomizu}, Theorem \ref{thm:dolbeault-invariant} and Theorem \ref{thm:bc-invariant}, we get the following result.
\begin{thm}\label{thm:bc-invariant-cor}
 Let $N=\left.\Gamma\right\backslash G$ be a nilmanifold endowed with a $G$-left-invariant complex structure $J$ and denote the Lie algebra naturally associated to $G$ by $\mathfrak{g}$. Suppose that one of the following conditions holds:
\begin{itemize}
 \item $N$ is holomorphically parallelizable;
 \item $J$ is an Abelian complex structure;
 \item $J$ is a nilpotent complex structure;
 \item $J$ is a rational complex structure;
 \item $\mathfrak{g}$ admits a torus-bundle series compatible with $J$ and with the rational structure induced by $\Gamma$.
\end{itemize}
Then the de Rham, Dolbeault, Bott-Chern and Aeppli cohomologies can be computed as the cohomologies of the corresponding subcomplexes given by the space of $G$-left-invariant forms on $N$; in other words, the inclusions of the several subcomplexes of $G$-left-invariant forms on $N$ into the corresponding complexes of forms on $N$ are quasi-isomorphisms:
$$ i\colon H^\bullet_{dR}\left(\mathfrak{g};\R\right) \;\stackrel{\simeq}{\hookrightarrow}\; H^\bullet_{dR}(N;\R) \qquad \text{ and }\qquad i\colon H^{\bullet,\bullet}_{\star}\left(\mathfrak{g}_\C\right) \;\stackrel{\simeq}{\hookrightarrow}\; H^{\bullet,\bullet}_{\star}\left(N\right) \;,$$
for $\star\in\{\del,\,\delbar,\,BC,\,A\}$.
\end{thm}

\medskip

A slight modification of \cite[Theorem 1]{console-fino} by S. Console and A. Fino gives the following result, which says that the property of computing the Bott-Chern cohomology using just left-invariant forms is open in the space of left-invariant complex structures on solvmanifolds.

\begin{thm}\label{thm:bc-invariant-open}
 Let $X=\left. \Gamma \right\backslash G$ be a solvmanifold endowed with a $G$-left-invariant complex structure $J$ and denote the Lie algebra naturally associated to $G$ by $\mathfrak{g}$. Let $\star\in\{\del,\,\delbar,\,BC,\,A\}$. Suppose that
$$ i\colon H^{\bullet,\bullet}_{\star_J}\left(\mathfrak{g}_\C\right) \;\stackrel{\simeq}{\hookrightarrow}\; H^{\bullet,\bullet}_{\star_J}\left(X\right) \;.$$
Then there exists an open neighbourhood $\mathcal{U}$ of $J$ in $\mathcal{C}\left(\mathfrak{g}\right)$ such that any $\tilde J\in \mathcal{U}$ still satisfies
$$ i\colon H^{\bullet,\bullet}_{\star_{\tilde{J}}}\left(\mathfrak{g}_\C\right) \;\stackrel{\simeq}{\hookrightarrow}\; H^{\bullet,\bullet}_{\star_{\tilde{J}}}\left(X\right) \;.$$
In other words, the set
$$ \mathcal{U} \;:=\; \left\{ J\in\mathcal{C}\left(\mathfrak{g}\right) \st i\colon H^{\bullet,\bullet}_{\star_J}\left(\mathfrak{g}_\C\right) \;\stackrel{\simeq}{\hookrightarrow}\; H^{\bullet,\bullet}_{\star_J}\left(X\right) \right\} $$
is open in $\mathcal{C}\left(\mathfrak{g}\right)$.
\end{thm}

\begin{proof}
 As a matter of notation, for $\varepsilon>0$ small enough, we consider
 $$ \left\{\left(X,\,J_t\right) \st t\in \Delta(0,\varepsilon)\right\} \twoheadrightarrow \Delta(0,\varepsilon) $$
 a complex-analytic family of $G$-left-invariant complex structures on $X$, where $\Delta(0,\varepsilon):=\left\{t\in\C^m\st \left|t\right|<\varepsilon\right\}$ for some $m\in\N\setminus\{0\}$; moreover, let $\left\{g_t\right\}_{t\in \Delta(0,\varepsilon)}$ be a family of $J_t$-Hermitian $G$-left-invariant metrics on $X$ depending smoothly on $t$. We will denote by $\delbar_t:=\delbar_{J_t}$ and $\delbar_t^*:=-*_{g_t}\del_{J_t}*_{g_t}$ the delbar operator and its $g_t$-adjoint respectively for the Hermitian structure $\left(J_t,\,g_t\right)$ and we set $\Delta_t:=\Delta_{\star_{J_t}}$ one of the differential operators involved in the definition of the Dolbeault, conjugate Dolbeault, Bott-Chern or Aeppli cohomologies with respect to $\left(J_t,\,g_t\right)$; we remark that $\Delta_t$ is a self-adjoint elliptic differential operator for all the considered cohomologies.

 By hypothesis, we have that $\left(H^{\bullet,\bullet}_{\star_{J_0}}\left(\mathfrak{g}_\C\right)\right)^\perp=\{0\}$, where the orthogonality is meant with respect to the scalar product induced by $g_0$, and we have to prove the same replacing $0$ with $t\in\Delta(0,\varepsilon)$. Therefore, it will suffice to prove that
$$ \Delta(0,\varepsilon)\ni t \mapsto \dim_\C\left(H^{\bullet,\bullet}_{\star_{J_t}}\left(\mathfrak{g}_\C\right)\right)^\perp\in\N $$
is an upper-semi-continuous function at $0$.\\
 For any $t\in\Delta(0,\varepsilon)$, being $\Delta_t$ a self-adjoint elliptic differential operator, there exists a complete orthonormal basis $\{e_i(t)\}_{i\in I}$ of eigenforms for $\Delta_t$ spanning $\left(\wedge^{\bullet,\bullet}_{J_t}\mathfrak{g}_\C^*\right)^\perp$, the orthogonal complement of the space of $G$-left-invariant forms, see \cite[Theorem 1]{kodaira-spencer-3}. For any $i\in I$ and $t\in\Delta(0,\varepsilon)$, let $a_i(t)$ be the eigenvalue corresponding to $e_i(t)$; $\Delta_t$ depending differentiably on $t\in\Delta(0,\varepsilon)$, for any $i\in I$, the function $\Delta(0,\varepsilon)\ni t \mapsto a_i(t)\in\C$ is continuous, see \cite[Theorem 2]{kodaira-spencer-3}. Therefore, for any $t_0\in\Delta(0,\varepsilon)$, choosing a constant $c>0$ such that $c\not\in \overline{\left\{a_i(t_0) \st i\in I\right\}}$, the function
 $$ \Psi_c\colon\Delta(0,\varepsilon)\to\N \;,\qquad t \mapsto \dim \Span\left\{e_i(t) \st a_i(t)<c\right\} $$
is locally constant at $t_0$; moreover, for any $t\in\Delta(0,\varepsilon)$ and for any $c>0$, we have
$$ \Psi_c(t) \;\geq\; \dim_\C \left(H^{\bullet,\bullet}_{\star_{J_t}}\left(\mathfrak{g}_\C\right)\right)^\perp \;.$$
Since the spectrum of $\Delta_{t_0}$ has no accumulation point for any $t_0\in\Delta(0,\varepsilon)$, see \cite[Theorem 1]{kodaira-spencer-3}, the theorem follows choosing $c>0$ small enough so that $\Psi_c(0)=\dim_\C\left(H^{\bullet,\bullet}_{\star_{J_0}}\left(\mathfrak{g}_\C\right)\right)^\perp$.
\end{proof}

\medskip

In particular, Theorem \ref{thm:bc-invariant-cor} and Theorem \ref{thm:bc-invariant-open} say that the following conjecture, which generalizes Conjecture \ref{conj:dolbeault}, holds for certain left-invariant complex structures on nilmanifolds.
\begin{conj}\label{conj:BC}
 Let $N=\left. \Gamma \right\backslash G$ be a nilmanifold endowed with a $G$-left-invariant complex structure $J$ and denote the Lie algebra naturally associated to $G$ by $\mathfrak{g}$. Then the de Rham, Dolbeault, Bott-Chern and Aeppli cohomologies can be computed as the cohomologies of the corresponding subcomplexes given by the space of $G$-left-invariant forms on $N$, that is,
$$ \dim_\R \left(H^\bullet_{dR}\left(\mathfrak{g};\R\right)\right)^{\perp} \;=\; 0 \qquad \text{ and }\qquad \dim_\C \left(H^{\bullet,\bullet}_{\star}\left(\mathfrak{g}_\C\right)\right)^{\perp}\;=\;0 \;,$$
where $\star\in\{\del,\,\delbar,\,BC,\,A\}$ and the orthogonality is meant with respect to the scalar product induced by a given $J$-Hermitian $G$-left-invariant metric $g$ on $N$.
\end{conj}

\section{The Iwasawa manifold and its small deformations}\label{part:iwasawa}

\subsection{The Iwasawa manifold}\label{sec:iwasawa}
Let $\mathbb{H}(3;\mathbb{C})$ be the $3$-dimensional \emph{Heisenberg group} over $\mathbb{C}$ defined by
$$
\mathbb{H}(3;\mathbb{C}) := \left\{
\left(
\begin{array}{ccc}
 1 & z^1 & z^3 \\
 0 &  1  & z^2 \\
 0 &  0  &  1
\end{array}
\right) \in \mathrm{GL}(3;\mathbb{C}) \st z^1,\,z^2,\,z^3 \in\C \right\}
\;,
$$
where the product is the one induced by matrix multiplication.
It is straightforward to prove that $\mathbb{H}(3;\C)$ is a connected simply-connected complex $2$-step nilpotent Lie group, that is, the Lie algebra $\left(\mathfrak{h}_3,\,\left[\sspace,\sspace\right]\right)$ naturally associated to $\mathbb{H}(3;\C)$ satisfies $\left[\mathfrak{h}_3,\mathfrak{h}_3\right]\neq 0$ and $\left[\mathfrak{h}_3,\left[\mathfrak{h}_3,\mathfrak{h}_3\right]\right]=0$.\\
One finds that
$$
\left\{
\begin{array}{rcl}
 \varphi^1 &:=& \de z^1 \\[5pt]
 \varphi^2 &:=& \de z^2 \\[5pt]
 \varphi^3 &:=& \de z^3-z^1\,\de z^2
\end{array}
\right.
$$
is a $\mathbb{H}(3;\C)$-left-invariant co-frame for the space of $(1,0)$-forms on $\mathbb{H}(3;\C)$ and that the structure equations with respect to this co-frame are
$$
\left\{
\begin{array}{rcl}
 \de\varphi^1 &=& 0 \\[5pt]
 \de\varphi^2 &=& 0 \\[5pt]
 \de\varphi^3 &=& -\varphi^1\wedge\varphi^2
\end{array}
\right. \;.
$$
For the sake of completeness, we write also
$$
\left\{
\begin{array}{rcl}
 \del\,\varphi^1 &=& 0 \\[5pt]
 \del\,\varphi^2 &=& 0 \\[5pt]
 \del\,\varphi^3 &=& -\varphi^1\wedge\varphi^2
\end{array}
\right. \qquad \text{ and } \qquad
\left\{
\begin{array}{rcl}
 \delbar\,\varphi^1 &=& 0 \\[5pt]
 \delbar\,\varphi^2 &=& 0 \\[5pt]
 \delbar\,\varphi^3 &=& 0
\end{array}
\right. \;.
$$

\smallskip

Consider the action on the left of $\mathbb{H}\left(3;\Z\left[\im\right]\right):=\mathbb{H}(3;\C)\cap\GL\left(3;\Z\left[\im\right]\right)$ on $\mathbb{H}\left(3;\C\right)$ and take the compact quotient
$$ \mathbb{I}_3 := \left. \mathbb{H}\left(3;\Z\left[\im\right]\right) \right\backslash \mathbb{H}(3;\C)\;. $$
One gets that $\mathbb{I}_3$ is a $3$-dimensional complex nilmanifold, whose ($\mathbb{H}(3;\C)$-left-invariant) complex structure $J_{\zero}$ is the one inherited by the standard complex structure on $\C^3$; $\mathbb{I}_3$ is called the \emph{Iwasawa manifold}.\\
Since the forms $\varphi^1$, $\varphi^2$ and $\varphi^3$ are $\mathbb{H}(3;\C)$-left-invariant, they define a co-frame also for $\left(T^{1,0}\I_3\right)^*$. Note that $\I_3$ is a holomorphically parallelizable manifold, that is, its holomorphic tangent bundle is holomorphically trivial.

Since, for example, $\varphi^3$ is a non-closed holomorphic form, it follows that $\I_3$ admits no K\"ahler metric. In fact, one can show that $\I_3$ is not formal, see \cite[page 158]{fernandez-gray}, therefore the underlying smooth manifold of $\I_3$ has no complex structure admitting K\"ahler metrics, even though all the topological obstructions concerning the Betti numbers are satisfied.

We sketch in Figure \ref{fig:iwasawa} the structure of the finite-dimensional double complex $\left(\wedge^{\bullet,\bullet}\left(\mathfrak{h}_3\otimes_\R\C\right)^*,\,\del,\,\delbar\right)$: horizontal arrows are meant as $\del$, vertical ones as $\delbar$ and zero arrows are not depicted.

\smallskip
\begin{center}
\begin{figure}[ht]
 \includegraphics[width=6.5cm]{iwasawa.1}
 \caption{The double complex $\left(\wedge^{\bullet,\bullet}\left(\mathfrak{h}_3\otimes_\R \C\right)^*,\,\del,\,\delbar\right)$.}
 \label{fig:iwasawa}
\end{figure}
\end{center}
\smallskip

\subsection{Small deformations of the Iwasawa manifold}\label{sec:deformations-iwasawa}	
I. Nakamura classified in \cite[\S2]{nakamura} the three-dimensional holomorphically parallelizable solvmanifolds into four classes by numerical invariants, giving $\mathbb{I}_3$ as an example in the second class. Moreover, he explicitly constructed the Kuranishi family of deformations of $\I_3$, showing that it is smooth, \cite[pages 94--95]{nakamura}, compare also \cite[Corollary 4.9]{rollenske-jems}; in particular, he computed the Hodge numbers of the small deformations of $\mathbb{I}_3$ proving that they have not to remain invariant along a complex-analytic family of complex structures, \cite[Theorem 2]{nakamura}, compare also \cite[\S4]{ye}; moreover, he proved in this way that the property of being holomorphically parallelizable is not stable under small deformations, \cite[page 86]{nakamura}, compare also \cite[Theorem 5.1, Corollary 5.2]{rollenske-jems}. I. Nakamura divided the small deformations of $\mathbb{I}_3$ into three classes according to their Hodge diamond, see \cite[page 96]{nakamura} (see 
also \cite[\S17]{ueno}).

\smallskip

\begin{thm}[{\cite[pages 94--96]{nakamura}}]
There exists a locally complete complex-analytic family of complex structures $\left\{X_\tempo=\left(\I_3,\,J_\tempo\right)\right\}_{\tempo\in \Delta(\zero,\varepsilon)}$, deformations of $\I_3$, depending on six parameters
$$ \tempo\;=\;\left(t_{11},\,t_{12},\,t_{21},\,t_{22},\,t_{31},\,t_{32}\right)\;\in\;\Delta(\zero,\varepsilon)\;\subset\;\C^6 \;,$$
where $\varepsilon>0$ is small enough, $\Delta(\zero,\varepsilon):=\left\{\mathbf{s}\in\C^6 \st \left|\mathbf{s}\right|<\varepsilon\right\}$ and $X_\zero=\I_3$.\\
A set of holomorphic coordinates for $X_\tempo$ is given by
$$
\left\{
\begin{array}{rcccl}
 \zeta^1 &:=& \zeta^1(\tempo) &:=& z^1+\sum_{k=1}^{2}t_{1k}\,\bar z^k \\[5pt]
 \zeta^2 &:=& \zeta^2(\tempo) &:=& z^2+\sum_{k=1}^{2}t_{2k}\,\bar z^k \\[5pt]
 \zeta^3 &:=& \zeta^3(\tempo) &:=& z^3+\sum_{k=1}^{2}\left(t_{3k}+t_{2k}\,z^1\right)\bar z^k+A\left(\bar z^1,\,\bar z^2\right)-D\left(\tempo\right)\,\bar z^3
\end{array}
\right.
$$
where
$$ D\left(\tempo\right):=\det\left(
\begin{array}{cc}
 t_{11} & t_{12} \\
 t_{21} & t_{22}
\end{array}
\right)
\qquad
$$
and
$$ A\left(\bar z^1,\,\bar z^2\right):=\frac{1}{2}\left(t_{11}\,t_{21}\,\left(\bar z^1\right)^2+2\,t_{11}\,t_{22}\,\bar z^1\,\bar z^2+t_{12}\,t_{22}\,\left(\bar z^2\right)^2\right) \;. $$
For every $\tempo\in \Delta(\zero,\varepsilon)$, the universal covering of $X_\tempo$ is $\C^3$; more precisely,
$$ X_\tempo \;=\; \left. \Gamma_\tempo \right \backslash\C^3 \;,$$
where $\Gamma_\tempo$ is the subgroup generated by the transformations
$$ \left(\zeta^1,\,\zeta^2,\,\zeta^3\right) \stackrel{\left(\omega^1,\,\omega^2,\,\omega^3\right)}{\mapsto} \left(\tilde\zeta^1,\,\tilde\zeta^2,\,\tilde\zeta^3\right) \;,$$
varying $\left(\omega^1,\,\omega^2,\,\omega^3\right)\in\left(\Z\left[\im\right]\right)^3$, where
$$
\left\{
\begin{array}{rcl} 
 \tilde\zeta^1 &:=& \zeta^1+\left(\omega^1+t_{11}\,\bar\omega^1+t_{12}\,\bar\omega^2\right) \\[10pt]
 \tilde\zeta^2 &:=& \zeta^2+\left(\omega^2+t_{21}\,\bar\omega^1+t_{22}\,\bar\omega^2\right) \\[10pt]
 \tilde\zeta^3 &:=& \zeta^3+\left(\omega^3+t_{31}\,\bar\omega^1+t_{32}\,\bar\omega^2\right)+\omega^1\,\zeta^2\\[5pt]
&&+\left(t_{21}\,\bar\omega^1+ t_{22}\,\bar\omega^2\right)\left(\zeta^1+\omega^1\right)+A\left(\bar \omega^1,\,\bar\omega^2\right)-D\left(\tempo\right)\,\bar\omega^3
\end{array}
\right. \;.
$$
\end{thm}

Bott-Chern cohomology distinguishes a finer classification of small deformations of $\I_3$ than Nakamura's classification into three classes, see \S\ref{sec:bott-chern-iwasawa}; the classes and subclasses of this classification are characterized by the following values of the parameters:\\
\begin{description}
 \item[class {\itshape (i)}] $t_{11}=t_{12}=t_{21}=t_{22}=0$; \\
 \item[class {\itshape (ii)}] $D\left(\tempo\right)=0$ and $\left(t_{11},\,t_{12},\,t_{21},\,t_{22}\right)\neq \left(0,\,0,\,0,\,0\right)$: 
    \begin{description}
     \item[subclass {\itshape (ii.a)}] $D\left(\tempo\right)=0$ and $\rk S=1$;
     \item[subclass {\itshape (ii.b)}] $D\left(\tempo\right)=0$ and $\rk S=2$;\\
    \end{description}
 \item[class {\itshape (iii)}] $D\left(\tempo\right)\neq 0$:
    \begin{description}
     \item[subclass {\itshape (iii.a)}] $D\left(\tempo\right)\neq 0$ and $\rk S=1$;
     \item[subclass {\itshape (iii.b)}] $D\left(\tempo\right)\neq 0$ and $\rk S=2$;\\
    \end{description}
\end{description}
the matrix $S$ is defined by
$$ S \;:=\;
\left(
\begin{array}{cccc}
 \overline{\sigma_{1\bar1}} & \overline{\sigma_{2\bar2}} & \overline{\sigma_{1\bar2}} & \overline{\sigma_{2\bar1}} \\
 \sigma_{1\bar1} & \sigma_{2\bar2} & \sigma_{2\bar1} & \sigma_{1\bar2}
\end{array}
\right)
$$
where $\sigma_{1\bar1},\,\sigma_{1\bar2},\,\sigma_{2\bar1},\,\sigma_{2\bar2}\in\C$ and $\sigma_{12}\in\C$ are complex numbers depending only on $\tempo$ such that
$$ \de\varphi^3_\tempo \;=:\; \sigma_{12}\,\varphi^1_{\tempo}\wedge\varphi^2_{\tempo}+\sigma_{1\bar1}\,\varphi^1_{\tempo}\wedge\bar\varphi^1_{\tempo}+\sigma_{1\bar2}\,\varphi^1_{\tempo}\wedge\bar\varphi^2_{\tempo}+\sigma_{2\bar1}\,\varphi^2_{\tempo}\wedge\bar\varphi^1_{\tempo}+\sigma_{2\bar2}\,\varphi^2_{\tempo}\wedge\bar\varphi^2_{\tempo} \;,$$
being
$$ \varphi^1_\tempo \;:=\; \de\zeta^1_{\tempo}\;,\quad \varphi^2_\tempo \;:=\; \de\zeta^2_{\tempo}\;,\quad \varphi^3_{\tempo} \;:=\; \de\zeta^3_{\tempo}-z_1\,\de\zeta^2_{\tempo}-\left(t_{21}\,\bar z^1+t_{22}\,\bar z^2\right)\de\zeta^1_{\tempo} \;,$$
see \S\ref{sec:structure-equations-iwasawa}.
The first order asymptotic behaviour of $\sigma_{12},\,\sigma_{1\bar1},\,\sigma_{1\bar2},\,\sigma_{2\bar1},\,\sigma_{2\bar2}$ for $\tempo$ near $0$ is the following:
\begin{equation}\label{eq:struttura-asintotica}
\left\{
\begin{array}{rcl}
\sigma_{12} &=& -1 +\opiccolo{\tempo} \\[5pt]
\sigma_{1\bar1} &=& t_{21} +\opiccolo{\tempo}  \\[5pt]
\sigma_{1\bar2} &=& t_{22} +\opiccolo{\tempo}  \\[5pt]
\sigma_{2\bar1} &=& -t_{11} +\opiccolo{\tempo} \\[5pt]
\sigma_{2\bar2} &=& -t_{12} +\opiccolo{\tempo} \\[5pt]
\end{array}
\right.
\qquad \text{ for } \qquad \tempo\in\text{ classes {\itshape (i)}, {\itshape (ii)} and {\itshape (iii)}}
\;,
\end{equation}
and, more precisely, for deformations in class {\itshape (ii)} we actually have that
\begin{equation}\label{eq:struttura-asintotica-ii}
\left\{
\begin{array}{rcl}
\sigma_{12} &=& -1 +\opiccolo{\tempo} \\[5pt]
\sigma_{1\bar1} &=& t_{21} \left(1+\opiccolouno\right)  \\[5pt]
\sigma_{1\bar2} &=& t_{22} \left(1+\opiccolouno\right)  \\[5pt]
\sigma_{2\bar1} &=& -t_{11} \left(1+\opiccolouno\right) \\[5pt]
\sigma_{2\bar2} &=& -t_{12} \left(1+\opiccolouno\right) \\[5pt]
\end{array}
\right. \qquad \text{ for } \qquad \tempo\in\text{ class {\itshape (ii)}}\;.
\end{equation}
\medskip

The complex manifold $X_\tempo$ is endowed with the $J_\tempo$-Hermitian $\mathbb{H}(3;\C)$-left-invariant metric $g_\tempo$, which is defined as follows:
$$ g_\tempo \;:=\; \sum_{j=1}^3\varphi_\tempo^j\odot\bar\varphi_\tempo^j \;. $$

\subsection{Computation of the structure equations for small deformations of the Iwasawa manifold}\label{sec:structure-equations-iwasawa}
In this section, we give the structure equations for the small deformations of the Iwasawa manifold; we will use these computations in \S\ref{sec:dolbeault-iwasawa} and \S\ref{sec:bott-chern-iwasawa}.

Consider
$$
\left\{
\begin{array}{rcl}
 \varphi^1_{\tempo} &:=& \de\zeta^1_{\tempo} \\[5pt]
 \varphi^2_{\tempo} &:=& \de\zeta^2_{\tempo} \\[5pt]
 \varphi^3_{\tempo} &:=& \de\zeta^3_{\tempo}-z_1\,\de\zeta^2_{\tempo}-\left(t_{21}\,\bar z^1+t_{22}\,\bar z^2\right)\de\zeta^1_{\tempo}
\end{array}
\right. \;.
$$
as a co-frame of $(1,0)$-forms on $X_{\tempo}$. We want to write the structure equations for $X_{\tempo}$ with respect to this co-frame.\\
For the complex structures in the class {\itshape (i)}, one checks that the structure equations are the same as the ones for $\I_3$, that is
$$
\left\{
\begin{array}{rcl}
 \de\varphi^1_{\tempo} &=& 0 \\[10pt]
 \de\varphi^2_{\tempo} &=& 0 \\[10pt]
 \de\varphi^3_{\tempo} &=& -\varphi^1_{\tempo}\wedge\varphi^2_{\tempo}
\end{array}
\right. \qquad \text{ for } \quad \tempo\in\text{ class {\itshape (i)}}
\;.
$$
For small deformations in classes {\itshape (ii)} and {\itshape (iii)}, we have that
$$
\left\{
\begin{array}{rcl}
 \de\varphi^1_{\tempo} &=& 0 \\[10pt]
 \de\varphi^2_{\tempo} &=& 0 \\[10pt]
 \de\varphi^3_{\tempo} &=& \sigma_{12}\,\varphi^1_{\tempo}\wedge\varphi^2_{\tempo}\\[5pt]
 &&+\sigma_{1\bar1}\,\varphi^1_{\tempo}\wedge\bar\varphi^1_{\tempo}+\sigma_{1\bar2}\,\varphi^1_{\tempo}\wedge\bar\varphi^2_{\tempo}\\[5pt]
 &&+\sigma_{2\bar1}\,\varphi^2_{\tempo}\wedge\bar\varphi^1_{\tempo}+\sigma_{2\bar2}\,\varphi^2_{\tempo}\wedge\bar\varphi^2_{\tempo}
\end{array}
\right. \qquad \text{ for } \qquad \tempo\in\text{ classes {\itshape (ii)} and {\itshape (iii)}} \;,
$$
where $\sigma_{12},\,\sigma_{1\bar1},\,\sigma_{1\bar2},\,\sigma_{2\bar1},\,\sigma_{2\bar2}\in\C$ are complex numbers depending only on $\tempo$; their explicit values in the case of class {\itshape (ii)} have been written down by the author and A. Tomassini in \cite[page 416]{angella-tomassini} and in the case of class {\itshape (iii)} are a bit more involved; here, we need just the first order asymptotic behavior for $\tempo$ near $0$ given in \eqref{eq:struttura-asintotica} and \eqref{eq:struttura-asintotica-ii}.

\section{The cohomologies of the Iwasawa manifold and of its small deformations}\label{part:cohomology-iwasawa}

\subsection{The de Rham cohomology of the Iwasawa manifold and of its small deformations}\label{sec:derham-iwasawa}
By a result by C. Ehresmann, every complex-analytic family of compact complex manifolds is trivial as a differentiable family of compact differentiable manifolds, see, e.g., \cite[Theorem 4.1]{kodaira-morrow}. Therefore the de Rham cohomology of small deformations of the Iwasawa manifold is the same as the de Rham cohomology of $\I_3$, which one can compute with the aid of Nomizu's theorem, \cite[Theorem 1]{nomizu}.\\
In the table below, we list the harmonic representatives with respect to the metric $g_\zero$ instead of their classes and we shorten the notation (as we will do also in the following) writing, for example, $\varphi^{A\bar B}:=\varphi^A\wedge\bar\varphi^B$.

\smallskip
\begin{center}
\begin{tabular}{c||c|c}
$H^k_{dR}\left(\I_3;\C\right)$ & $g_{\zero}$-harmonic representatives & dimension\\[5pt]
\hline
\hline
$k=1$ & $\varphi^1$, $\varphi^2$, $\bar{\varphi}^1$, $\bar{\varphi}^2$ & $4$\\[5pt]
$k=2$ & $\varphi^{13}$, $\varphi^{23}$, $\varphi^{1\bar{1}}$, $\varphi^{1\bar{2}}$, $\varphi^{2\bar{1}}$, $\varphi^{2\bar{2}}$,$\varphi^{\bar{1}\bar{3}}$, $\varphi^{\bar{2}\bar{3}}$ & $8$\\[5pt]
$k=3$ & $\varphi^{123}$, $\varphi^{13\bar{1}}$, $\varphi^{13\bar{2}}$, $\varphi^{23\bar{1}}$, $\varphi^{23\bar{2}}$, $\varphi^{1\bar{1}\bar{3}}$, $\varphi^{1\bar{2}\bar{3}}$, $\varphi^{2\bar{1}\bar{3}}$, $\varphi^{2\bar{2}\bar{3}}$, $\varphi^{\bar{1}\bar{2}\bar{3}}$ & $10$\\[5pt]
$k=4$ & $\varphi^{123\bar{1}}$, $\varphi^{123\bar{2}}$, $\varphi^{13\bar{1}\bar{3}}$, $\varphi^{13\bar{2}\bar{3}}$, $\varphi^{23\bar{1}\bar{3}}$, $\varphi^{23\bar{2}\bar{3}}$, $\varphi^{1\bar{1}\bar{2}\bar{3}}$, $\varphi^{2\bar{1}\bar{2}\bar{3}}$ & $8$\\[5pt]
$k=5$ & $\varphi^{123\bar{1}\bar{3}}$, $\varphi^{123\bar{2}\bar{3}}$, $\varphi^{13\bar{1}\bar{2}\bar{3}}$, $\varphi^{23\bar{1}\bar{2}\bar{3}}$ & $4$
\end{tabular}
\end{center}
\smallskip

Note that all the $g_\zero$-harmonic representatives of $H^\bullet_{dR}(\mathbb{I}_3;\R)$ are of pure type with respect to $J_{\zero}$, that is, they are in $\left(\wedge^{p,q}\mathbb{I}_3\oplus \wedge^{q,p}\mathbb{I}_3\right) \cap \wedge^{p+q}\mathbb{I}_3$ for some $p,q\in\{0,\,1,\,2,\,3\}$. This holds no more true for $J_\tempo$ with $\tempo\neq\zero$ small enough, see \cite[Theorem 5.1]{angella-tomassini}.

\subsection{The Dolbeault cohomology of the Iwasawa manifold and of its small deformations}\label{sec:dolbeault-iwasawa}
The Hodge numbers of the Iwasawa manifold and of its small deformations have been computed by I. Nakamura in \cite[page 96]{nakamura}. The $g_\tempo$-harmonic representatives for $H^{\bullet,\bullet}_{\delbar}\left(X_{\tempo}\right)$, for $\tempo$ small enough, can be computed using the considerations in \S\ref{subsec:cohomology-computation-derham-dolbeault} and the structure equations given in \S\ref{sec:structure-equations-iwasawa}. We collect here the results of the computations.\\
In order to reduce the number of cases under consideration, note that, on a compact complex Hermitian manifold $X$ of complex dimension $n$, for any $p,q\in\N$, the Hodge-$*$-operator induces an isomorphism between
$$ H^{p,q}_{\delbar}(X) \stackrel{\simeq}{\longrightarrow} H^{n-q,n-p}_{\del}(X)\simeq \overline{H^{n-p,n-q}_{\delbar}(X)} \;.$$

\medskip

As regards $1$-forms, we have that
$$
\dim_\C H^{1,0}_{\delbar}(X_\tempo)\;=\;
\left\{
\begin{array}{lcl}
3 & \quad \text{ for } \quad & \tempo\in\text{ class {\itshape (i)}}\\[5pt]
2 & \quad \text{ for } \quad & \tempo\in\text{ classes {\itshape (ii)} and {\itshape (iii)}}
\end{array}
\right. \;,
$$
which in particular means that $X_{\tempo}$ is not holomorphically parallelizable for $\tempo$ in classes {\itshape (ii)} and {\itshape (iii)}, see \cite[pages 86, 96]{nakamura},
and that
$$
\dim_\C H^{0,1}_{\delbar}(X_\tempo)\;=\;2 \quad \text{ for } \quad \tempo\in\text{ classes {\itshape (i)}, {\itshape (ii)} and {\itshape (iii)}} \;.
$$

\medskip

\noindent As regards $2$-forms, we have that
$$
\dim_\C H^{2,0}_{\delbar}(X_\tempo)\;=\;
\left\{
\begin{array}{lcl}
 3 & \quad \text{ for } \quad & \tempo\in\text{ class {\itshape (i)}} \\[5pt]
 2 & \quad \text{ for } \quad & \tempo\in\text{ class {\itshape (ii)}} \\[5pt]
 1 & \quad \text{ for } \quad & \tempo\in\text{ class {\itshape (iii)}}
\end{array}
\right.
$$
and
$$
\dim_\C H^{1,1}_{\delbar}(X_\tempo)\;=\;
\left\{
\begin{array}{lcl}
 6 & \quad \text{ for } \quad & \tempo\in\text{ class {\itshape (i)}} \\[5pt]
 5 & \quad \text{ for } \quad & \tempo\in\text{ classes {\itshape (ii)} and {\itshape (iii)}}
\end{array}
\right.
$$
and
$$
\dim_\C H^{0,2}_{\delbar}(X_\tempo)\;=\;2 \quad \text{ for } \quad \tempo\in\text{ classes {\itshape (i)}, {\itshape (ii)} and {\itshape (iii)}} \;.
$$

\medskip

\noindent Finally, as regards $3$-forms, we have that
$$
\dim_\C H^{3,0}_{\delbar}(X_{\tempo}) \;=\; 1 \quad \text{ for } \quad \tempo\in\text{ classes {\itshape (i)}, {\itshape (ii)} and {\itshape (iii)}}
$$
and
$$
\dim_\C H^{2,1}_{\delbar}(X_{\tempo})\;=\;
\left\{
\begin{array}{lcl}
6 & \quad \text{ for } \quad & \tempo\in\text{ class {\itshape (i)}} \\[5pt]
5 & \quad \text{ for } \quad & \tempo\in\text{ class {\itshape (ii)}} \\[5pt]
4 & \quad \text{ for } \quad & \tempo\in\text{ class {\itshape (iii)}}
\end{array}
\right. \;.
$$

\subsection{The Bott-Chern and Aeppli cohomologies of the Iwasawa manifold and of its small deformations}\label{sec:bott-chern-iwasawa}
In this section, using the results of Theorem \ref{thm:bc-invariant-cor} and Theorem \ref{thm:bc-invariant-open}, we explicitly compute the dimensions and we compute the $g_\tempo$-harmonic representatives for $H_{BC}^{\bullet,\bullet}(X_{\tempo})$, for $\tempo$ small enough.
\\
In order to reduce the number of cases under consideration, note that, on a compact complex Hermitian manifold $X$ of complex dimension $n$, for every $p,q\in\N$, the conjugation induces an isomorphism between
$$ H^{p,q}_{BC}(X) \;\simeq\; H^{q,p}_{BC}(X) $$
and the Hodge-$*$-operator induces an isomorphism between
$$ H^{p,q}_{BC}(X) \stackrel{\simeq}{\longrightarrow} H^{n-q,n-p}_{A}(X) \;;$$
furthermore, note that
$$ H^{p,0}_{BC}(X) \;\simeq\; \ker\left(\de\colon \wedge^{p,0}X\to\wedge^{p+1}(X;\C)\right) $$
and
$$ H^{n,0}_{BC}(X) \;\simeq\; H^{n,0}_{\delbar}(X) \;.$$

\medskip

We summarize the results of the computations below in the following theorem.
\begin{thm}\label{thm:bc-iwasawa}
 Let $\I_3$ be the Iwasawa manifold and consider its small deformations $\left\{X_\tempo=\left(\I_3,\,J_\tempo\right)\right\}_{\tempo\in\Delta(\zero,\varepsilon)}$, where $\varepsilon>0$ is small enough and $X_\zero=\I_3$. Then the dimensions $h^{p,q}_{BC}:=h^{p,q}_{BC}\left(X_\tempo\right):=\dim_\C H^{p,q}_{BC}\left(X_\tempo\right)=\dim_\C H^{3-p,3-q}_{A}\left(X_\tempo\right)$ does not depend on $\tempo\in\Delta(\zero,\varepsilon)$ if $p+q$ is odd or $(p,q)\in\left\{(1,1),\,(3,1),\,(1,3)\right\}$ and they are equal to
$$
\begin{array}{ccccccccccc}
 h^{1,0}_{BC} &=& h^{0,1}_{BC} &=& 2 \;, & & & & & & \\[5pt]
 h^{2,0}_{BC} &=& h^{0,2}_{BC} &\in& \left\{1,\,2,\,3\right\}\;, & \qquad & & & h^{1,1}_{BC} & = & 4\;, \\[5pt]
 h^{3,0}_{BC} &=& h^{0,3}_{BC} &=& 1 \;, & \qquad & h^{2,1}_{BC} & = & h^{1,2}_{BC} & = & 6\;, \\[5pt]
 h^{3,1}_{BC} &=& h^{1,3}_{BC} &=& 2 \;, & \qquad & & & h^{2,2}_{BC} & \in & \left\{6,\,7,\,8\right\} \;, \\[5pt]
 h^{3,2}_{BC} &=& h^{2,3}_{BC} &=& 3 \;. & & & & & &
\end{array}
$$
\end{thm}

\begin{rem}
 As a consequence of the computations below, we notice that the Bott-Chern cohomology allows to give a finer classification of the small deformations of $\I_3$ than the Dolbeault cohomology: indeed, note that $\dim_\C H^{2,2}_{BC}(X_\tempo)$ assumes different values according to different parameters in the same class {\itshape (ii)} or {\itshape (iii)}; in a sense, this says that the Bott-Chern cohomology ``carries more informations'' about the complex structure that the Dolbeault one. Note also that most of the dimensions of Bott-Chern cohomology groups are invariant under small deformations: this happens for example for the odd-degree Bott-Chern cohomology groups.
\end{rem}

\medskip

\subsubsection*{$1$-forms}\label{subsec:bott-chern-iwasawa-1}
It is straightforward to check that
$$ H^{1,0}_{BC}(X_{\tempo})\;=\:\C\left\langle \phit{1},\,\phit{2}\right\rangle \quad \text{ for } \quad \tempo\in\text{ classes {\itshape (i)}, {\itshape (ii)} and {\itshape (iii)}} \;. $$

\subsubsection*{$2$-forms}\label{subsec:bott-chern-iwasawa-2}
It is straightforward to compute
$$ H^{2,0}_{BC}(X_{\tempo}) \;=\; \C\left\langle \phit{12},\,\phit{13},\,\phit{23} \right\rangle \quad \text{ for }\quad \tempo\in\text{ class {\itshape (i)}} \;.$$
The computations for $H^{2,0}_{BC}(X_{\tempo})$ reduce to find $\psi=A\,\phit{12}+B\,\phit{13}+C\,\phit{23}$ where $A,\,B,\,C\in\C$ satisfy the linear system
$$
\left(
\begin{array}{ccc}
 0 & 0 & 0 \\
 0 & -\sigma_{2\bar1} & \sigma_{1\bar1} \\
 0 & -\sigma_{2\bar2} & \sigma_{1\bar2}
\end{array}
\right)
\cdot
\left(
\begin{array}{c}
 A \\
 B \\
 C
\end{array}
\right)
=
\left(
\begin{array}{c}
 0 \\
 0 \\
 0
\end{array}
\right) \;,
$$
whose matrix has rank $0$ for $\tempo\in\text{ class {\itshape (i)}}$, rank $1$ for $\tempo\in\text{ class {\itshape (ii)}}$ and rank $2$ for $\tempo\in\text{ class {\itshape (iii)}}$; so, in particular, we get that
$$ \dim_\C H^{2,0}_{BC}(X_{\tempo}) \;=\; 2 \quad \text{ for } \quad \tempo\in\text{ class {\itshape (ii)}} $$
and
$$ \dim_\C H^{2,0}_{BC}(X_{\tempo}) \;=\; 1 \quad \text{ for } \quad \tempo\in\text{ class {\itshape (iii)}} $$
(more precisely, for $\tempo\in\text{ class {\itshape (iii)}}$ we have $H^{2,0}_{BC}(X_{\tempo})=\C\left\langle \phit{12}\right\rangle$).\\
It remains to compute $H^{1,1}_{BC}(X_\tempo)$ for $\tempo\in\text{ classes {\itshape (i)}, {\itshape (ii)} and {\itshape (iii)}}$. First of all, it is easy to check that
$$ H^{1,1}_{BC}(X_{\tempo}) \;\supseteq\; \C\left\langle \phit{1\bar1},\,\phit{1\bar2},\,\phit{2\bar1},\,\phit{2\bar2} \right\rangle \quad \text{ for } \quad \tempo\in\text{ classes {\itshape (i)}, {\itshape (ii)} and {\itshape (iii)}} \;,$$
and equality holds if $\tempo\in\text{ class {\itshape (i)}}$, hence, in particular, if $\tempo=\zero$: this immediately implies that
$$ H^{1,1}_{BC}(X_{\tempo}) \;=\; \C\left\langle \phit{1\bar1},\,\phit{1\bar2},\,\phit{2\bar1},\,\phit{2\bar2} \right\rangle \quad \text{ for } \quad \tempo\in\text{ classes {\itshape (i)}, {\itshape (ii)} and {\itshape (iii)}} \;:$$
indeed, being $H^{1,1}_{BC}(X_\tempo)$ isomorphic to the kernel of the self-adjoint elliptic differential operator $\tilde\Delta_{BC_{J_\tempo}}$, the function $\tempo\mapsto \dim_\C H^{1,1}_{BC}(X_\tempo)$ is upper-semi-continuous at $0$. (One can explain this argument saying that the new parts appearing in the computations for $\tempo\neq \zero$ are ``too small'' to balance out the lack for the $\del$-closure or the $\delbar$-closure.) From another point of view, we can note that $(1,1)$-forms of the type $\psi=A\,\phit{1\bar3}+B\,\phit{2\bar3}+C\,\phit{3\bar1}+D\,\phit{3\bar2}$ are $\tilde\Delta_{BC_{J_\tempo}}$-harmonic if and only if $A,\,B,\,C,\,D\in\C$ satisfy the linear system
$$
\left(
\begin{array}{cccc}
 -\overline{\sigma_{12}} & 0 & -\sigma_{1\bar2} & -\sigma_{1\bar1} \\
 0 & -\overline{\sigma_{12}} & -\sigma_{2\bar2} & -\sigma_{2\bar1} \\
\hline
 \overline{\sigma_{1\bar2}} & -\overline{\sigma_{1\bar1}} & \sigma_{12} & 0 \\
 \overline{\sigma_{2\bar2}} & -\overline{\sigma_{2\bar1}} & 0 & \sigma_{12} 
\end{array}
\right)
\cdot
\left(
\begin{array}{c}
 A \\
 B \\
 C \\
 D
\end{array}
\right)
\;=\;
\left(
\begin{array}{c}
 0 \\
 0 \\
 0 \\
 0
\end{array}
\right) \;,
$$
whose matrix has rank $4$ for every $\tempo\in\text{ classes {\itshape (i)}, {\itshape (ii)} and {\itshape (iii)}}$.

\subsubsection*{$3$-forms}\label{subsec:bott-chern-iwasawa-3}
Since the special form of the structure equations, it is straightforward to compute
$$ H^{3,0}_{BC}(X_{\tempo})\;=\;\C \left\langle \phit{123} \right\rangle \quad \text{ for } \quad \tempo\in\text{ classes {\itshape (i)}, {\itshape (ii)} and {\itshape (iii)}} \;.$$
Moreover,
\begin{eqnarray*}
H^{2,1}_{BC}(X_{\tempo})&=&\C \left\langle \phit{12\bar1},\, \phit{12\bar2},\, \phit{13\bar1}-\frac{\sigma_{2\bar2}}{\overline{\sigma_{12}}}\,\phit{12\bar3},\, \phit{13\bar2}+\frac{\sigma_{2\bar1}}{\overline{\sigma_{12}}}\,\phit{12\bar3},\, \phit{23\bar1}+\frac{\sigma_{1\bar2}}{\overline{\sigma_{12}}}\,\phit{12\bar3},\right.\\[5pt]
&&\left. \phit{23\bar2}-\frac{\sigma_{1\bar1}}{\overline{\sigma_{12}}}\,\phit{12\bar3} \right\rangle \quad \text{ for } \quad \tempo\in\text{ classes {\itshape (i)}, {\itshape (ii)} and {\itshape (iii)}}\;;
\end{eqnarray*}
in particular,
$$ H^{2,1}_{BC}(X_{\tempo})\;=\;\C \left\langle \phit{12\bar1},\, \phit{12\bar2},\, \phit{13\bar1},\, \phit{13\bar2},\, \phit{23\bar1},\, \phit{23\bar2} \right\rangle  \quad \text{ for } \quad \tempo\in\text{ class {\itshape (i)}} \;.$$
From another point of view, one can easily check that
$$
H^{2,1}_{BC}(X_{\tempo}) \;\supseteq\; \C \left\langle \phit{12\bar1},\, \phit{12\bar2} \right\rangle \quad \text{ for } \quad \tempo\in\text{ classes {\itshape (i)}, {\itshape (ii)} and {\itshape (iii)}}
$$
and that the $(2,1)$-forms of the type $\psi=A\,\phit{12\bar3}+B\,\phit{13\bar1}+C\,\phit{13\bar2}+D\,\phit{23\bar1}+E\,\phit{23\bar2}$ are $\tilde\Delta_{BC_{J_\tempo}}$-harmonic if and only if $A,\,B,\,C,\,D,\,E\in\C$ satisfy the equation
$$
\left(
\begin{array}{ccccc}
 \overline{\sigma_{12}} & \sigma_{2\bar2} & -\sigma_{2\bar1} & \sigma_{1\bar2} & \sigma_{1\bar1}
\end{array}
\right)
\cdot
\left(
\begin{array}{c}
 A \\
 B \\
 C \\
 D \\
 E
\end{array}
\right)
\;=\;
 0 \;,
$$
whose matrix has rank $1$ for every $\tempo\in\text{ classes {\itshape (i)}, {\itshape (ii)} and {\itshape (iii)}}$. Note in particular that the dimensions of $H^{3,0}_{BC}(X_{\tempo})$ and of $H^{2,1}_{BC}(X_{\tempo})$ do not depend on $\tempo$.

\subsubsection*{$4$-forms}\label{subsec:bott-chern-iwasawa-4}
It is straightforward to compute
$$ H^{3,1}_{BC}(X_\tempo) \;=\; \C \left\langle \phit{123\bar1},\,\phit{123\bar2} \right\rangle \quad \text{ for } \quad \tempo\in\text{ classes {\itshape (i)}, {\itshape (ii)} and {\itshape (iii)}} $$
and
\begin{eqnarray*}
H^{2,2}_{BC}(X_\tempo) &=& \C\left\langle \phit{12\bar1\bar3},\, \phit{12\bar2\bar3},\, \phit{13\bar1\bar2},\, \phit{13\bar1\bar3},\, \phit{13\bar2\bar3},\, \phit{23\bar1\bar2},\right.\\[5pt]
&& \left.\phit{23\bar1\bar3},\, \phit{23\bar2\bar3} \right\rangle \quad \text{ for } \quad \tempo\in\text{ class {\itshape (i)}} \;.
\end{eqnarray*}
Moreover, one can check that
$$ H^{2,2}_{BC}(X_\tempo) \;\supseteq\; \C\left\langle \phit{12\bar1\bar3},\, \phit{12\bar2\bar3},\, \phit{13\bar1\bar2},\, \phit{23\bar1\bar2} \right\rangle \quad \text{ for } \quad \tempo\in\text{ classes {\itshape (i)}, {\itshape (ii)} and {\itshape (iii)}} \;.$$
For $H^{2,2}_{BC}(X_\tempo)$ with $\tempo\in\text{ class {\itshape (ii)} and {\itshape (iii)}}$, we get a new behavior: there are subclasses in both class {\itshape (ii)} and class {\itshape (iii)}, which can be distinguished by the dimension of $H^{2,2}_{BC}(X_\tempo)$. Indeed, consider $(2,2)$-forms of the type $\psi=A\,\phit{13\bar1\bar3}+B\,\phit{13\bar2\bar3}+C\,\phit{23\bar1\bar3}+D\,\phit{23\bar2\bar3}$; a straightforward computation shows that such a $\psi$ is $\tilde\Delta_{BC_{J_\tempo}}$-harmonic if and only if $A,\,B,\,C,\,D\in\C$ satisfy the linear system
$$
\left(
\begin{array}{cccc}
 \overline{\sigma_{2\bar2}} & -\overline{\sigma_{1\bar2}} & -\overline{\sigma_{2\bar1}} & \overline{\sigma_{1\bar1}} \\
 \sigma_{2\bar2} & -\sigma_{2\bar1} & -\sigma_{1\bar2} & \sigma_{1\bar1}
\end{array}
\right)
\cdot
\left(
\begin{array}{c}
 A \\
 B \\
 C \\
 D
\end{array}
\right)
=
\left(
\begin{array}{c}
 0 \\
 0
\end{array}
\right) \;:
$$
as one can easily note, the rank of the matrix involved is $0$ for $\tempo\in\text{ class {\itshape (i)}}$, while it is $1$ or $2$ depending on the values of the parameters and independently from $\tempo$ belonging to class {\itshape (ii)} or class {\itshape (iii)}. Therefore
$$ \dim_\C H^{2,2}_{BC}(X_\tempo) \;=\; 7 \quad \text{ for } \quad \tempo\in\text{ subclasses {\itshape (ii.a)} and {\itshape (iii.a)}} $$
and
$$ \dim_\C H^{2,2}_{BC}(X_\tempo) \;=\; 6 \quad \text{ for } \quad \tempo\in\text{ subclasses {\itshape (ii.b)} and {\itshape (iii.b)}} \;.$$

\subsubsection*{$5$-forms}\label{subsec:bott-chern-iwasawa-5}
Lastly, let us compute $H^{3,2}_{BC}(X_\tempo)$.\\
It is straightforward to check that
$$ H^{3,2}_{BC}(X_{\tempo}) \;=\; \C\left\langle \phit{123\bar1\bar2},\, \phit{123\bar1\bar3},\, \phit{123\bar2\bar3} \right\rangle \quad \text{ for } \quad \tempo\in\text{ classes {\itshape (i)}, {\itshape (ii)} and {\itshape (iii)}} \;:$$
in particular, it does not depend on $\tempo\in\Delta(\zero,\varepsilon)$.

\appendix

\begin{center}

\begin{landscape}
\section{Dimensions of the cohomologies of the Iwasawa manifold and of its small deformations}\label{app:chart}

\begin{center}
\begin{tabular}{c|ccccc}
$\mathbf{H^\bullet_{dR}}$ & $\mathbf{b_1}$ & $\mathbf{b_2}$ & $\mathbf{b_3}$ & $\mathbf{b_4}$ & $\mathbf{b_5}$ \\[5pt]
\hline
$\mathbb{I}_3$ and {\itshape (i)}, {\itshape (ii)}, {\itshape (iii)} & $4$ & $8$ & $10$ & $8$ & $4$
\end{tabular} 
\end{center}

\vspace{12pt}

\begin{center}
\begin{tabular*}{16cm}{c|cc|ccc|cccc|ccc|cc}
$\mathbf{H^{\bullet\bullet}_{\delbar}}$ & $\mathbf{h^{1,0}_{\delbar}}$ & $\mathbf{h^{0,1}_{\delbar}}$ & $\mathbf{h^{2,0}_{\delbar}}$ & $\mathbf{h^{1,1}_{\delbar}}$ & $\mathbf{h^{0,2}_{\delbar}}$ & $\mathbf{h^{3,0}_{\delbar}}$ & $\mathbf{h^{2,1}_{\delbar}}$ & $\mathbf{h^{1,2}_{\delbar}}$ & $\mathbf{h^{0,3}_{\delbar}}$ & $\mathbf{h^{3,1}_{\delbar}}$ & $\mathbf{h^{2,2}_{\delbar}}$ & $\mathbf{h^{1,3}_{\delbar}}$ & $\mathbf{h^{3,2}_{\delbar}}$ & $\mathbf{h^{2,3}_{\delbar}}$ \\[5pt]
\hline
$\I_3$ and \itshape{(i)} & $3$ & $2$ & $3$ & $6$ & $2$ & $1$ & $6$ & $6$ & $1$ & $2$ & $6$ & $3$ & $2$ & $3$ \\[5pt]
\itshape{(ii)} & $2$ & $2$ & $2$ & $5$ & $2$ & $1$ & $5$ & $5$ & $1$ & $2$ & $5$ & $2$ & $2$ & $2$ \\[5pt]
\itshape{(iii)} & $2$ & $2$ & $1$ & $5$ & $2$ & $1$ & $4$ & $4$ & $1$ & $2$ & $5$ & $1$ & $2$ & $2$
\end{tabular*}
\end{center}

\vspace{12pt}

\begin{center}
\begin{tabular*}{16cm}{c|cc|ccc|cccc|ccc|cc}
$\mathbf{H^{\bullet\bullet}_{\textrm{BC}}}$ & $\mathbf{h^{1,0}_{\textrm{BC}}}$ & $\mathbf{h^{0,1}_{\textrm{BC}}}$ & $\mathbf{h^{2,0}_{\textrm{BC}}}$ & $\mathbf{h^{1,1}_{\textrm{BC}}}$ & $\mathbf{h^{0,2}_{\textrm{BC}}}$ & $\mathbf{h^{3,0}_{\textrm{BC}}}$ & $\mathbf{h^{2,1}_{\textrm{BC}}}$ & $\mathbf{h^{1,2}_{\textrm{BC}}}$ & $\mathbf{h^{0,3}_{\textrm{BC}}}$ & $\mathbf{h^{3,1}_{\textrm{BC}}}$ & $\mathbf{h^{2,2}_{\textrm{BC}}}$ & $\mathbf{h^{1,3}_{\textrm{BC}}}$ & $\mathbf{h^{3,2}_{\textrm{BC}}}$ & $\mathbf{h^{2,3}_{\textrm{BC}}}$ \\[5pt]
\hline
$\I_3$ and \itshape{(i)} & $2$ & $2$ & $3$ & $4$ & $3$ & $1$ & $6$ & $6$ & $1$ & $2$ & $8$ & $2$ & $3$ & $3$ \\[5pt]
\itshape{(ii.a)} & $2$ & $2$ & $2$ & $4$ & $2$ & $1$ & $6$ & $6$ & $1$ & $2$ & $7$ & $2$ & $3$ & $3$ \\[5pt]
\itshape{(ii.b)} & $2$ & $2$ & $2$ & $4$ & $2$ & $1$ & $6$ & $6$ & $1$ & $2$ & $6$ & $2$ & $3$ & $3$ \\[5pt]
\itshape{(iii.a)} & $2$ & $2$ & $1$ & $4$ & $1$ & $1$ & $6$ & $6$ & $1$ & $2$ & $7$ & $2$ & $3$ & $3$ \\[5pt]
\itshape{(iii.b)} & $2$ & $2$ & $1$ & $4$ & $1$ & $1$ & $6$ & $6$ & $1$ & $2$ & $6$ & $2$ & $3$ & $3$
\end{tabular*}
\end{center}

\vspace{12pt}

\begin{center}
\begin{tabular*}{16cm}{c|cc|ccc|cccc|ccc|cc}
$\mathbf{H^{\bullet\bullet}_{\textrm{A}}}$ & $\mathbf{h^{1,0}_{\textrm{A}}}$ & $\mathbf{h^{0,1}_{\textrm{A}}}$ & $\mathbf{h^{2,0}_{\textrm{A}}}$ & $\mathbf{h^{1,1}_{\textrm{A}}}$ & $\mathbf{h^{0,2}_{\textrm{A}}}$ & $\mathbf{h^{3,0}_{\textrm{A}}}$ & $\mathbf{h^{2,1}_{\textrm{A}}}$ & $\mathbf{h^{1,2}_{\textrm{A}}}$ & $\mathbf{h^{0,3}_{\textrm{A}}}$ & $\mathbf{h^{3,1}_{\textrm{A}}}$ & $\mathbf{h^{2,2}_{\textrm{A}}}$ & $\mathbf{h^{1,3}_{\textrm{A}}}$ & $\mathbf{h^{3,2}_{\textrm{A}}}$ & $\mathbf{h^{2,3}_{\textrm{A}}}$ \\[5pt]
\hline
$\I_3$ and \itshape{(i)} & $3$ & $3$ & $2$ & $8$ & $2$ & $1$ & $6$ & $6$ & $1$ & $3$ & $4$ & $3$ & $2$ & $2$ \\[5pt]
\itshape{(ii.a)} & $3$ & $3$ & $2$ & $7$ & $2$ & $1$ & $6$ & $6$ & $1$ & $2$ & $4$ & $2$ & $2$ & $2$ \\[5pt]
\itshape{(ii.b)} & $3$ & $3$ & $2$ & $6$ & $2$ & $1$ & $6$ & $6$ & $1$ & $2$ & $4$ & $2$ & $2$ & $2$ \\[5pt]
\itshape{(iii.a)} & $3$ & $3$ & $2$ & $7$ & $2$ & $1$ & $6$ & $6$ & $1$ & $1$ & $4$ & $1$ & $2$ & $2$ \\[5pt]
\itshape{(iii.b)} & $3$ & $3$ & $2$ & $6$ & $2$ & $1$ & $6$ & $6$ & $1$ & $1$ & $4$ & $1$ & $2$ & $2$
\end{tabular*}
\end{center}

\end{landscape}

\end{center}

\end{document}